\documentclass[11pt,a4paper,twoside]{article}
\usepackage{amsfonts} 
\usepackage{amssymb}
\usepackage{amsmath, amsthm, array, ifthen}
\usepackage{mathrsfs} 
\usepackage{caption} 
\usepackage{latexsym}
\usepackage{comment}
\usepackage{mathtools,nccmath}

\DeclarePairedDelimiterX\Ioo[2]{\rbrack}{\lbrack}{#1,#2}
\DeclarePairedDelimiterX\Iof[2]{\rbrack}{\rbrack}{#1,#2}
\DeclarePairedDelimiterX\Ifo[2]{\lbrack}{\lbrack}{#1,#2}
\DeclarePairedDelimiterX\Iff[2]{\lbrack}{\rbrack}{#1,#2}
\DeclarePairedDelimiterX\Jfo[2]{\lbrack}{\lbrack}{#1;#2}
\DeclarePairedDelimiterX\Jff[2]{\lbrack}{\rbrack}{#1;#2}

\usepackage{enumerate}

\usepackage[np]{numprint} 
\npdecimalsign{\ensuremath{.}}

\usepackage{enumitem}

\usepackage{xfrac}
\usepackage{graphicx}

\usepackage{lmodern}

\usepackage[utf8]{inputenc} 

\usepackage{color}

\usepackage[colorlinks]{hyperref}

\newcommand{\RR}{\mathbb{R}}

\newcommand{\ZZ}{\mathbb{Z}}

\newcommand{\tri}{\begin{scriptsize}$\triangleright$\end{scriptsize}}

\newcommand{\mell}[2][f]{\mathcal{M}(f)} 

\newcommand{\diff}{\mathop{}\mathopen{}\mathrm{d}}

\newcommand{\sdfrac}[2]{\mbox{\small$\displaystyle\frac{#1}{#2}$}}

\theoremstyle{theorem}
\newtheorem{prop}{Proposition}
\newtheorem{prop/not}{Proposition/notation}
\newtheorem{thm}{Theorem}

\newtheorem{cor}{Corollary}
\newtheorem{lem}{Lemma}

\theoremstyle{remark}
\newtheorem{rem}{Remark}

\newcommand{\ioe}{\leqslant}
\newcommand{\soe}{\geqslant}

\usepackage[usenames,dvipsnames,svgnames,table]{xcolor}

\DeclareMathOperator{\re}{Re}
\DeclareMathOperator{\im}{Im}

\providecommand{\keywordsubject}[1]{\textbf{2020 Mathematics Subject Classification :} \: #1}
\providecommand{\keywords}[1]{\textbf{Keywords :} #1}

\title{On a Smoothed Dirichlet Divisor Problem}

\usepackage[left=2.5cm,right=2.5cm,top=1.5cm,bottom=1.5cm]{geometry}
\author{ Olivier Bordellès and Florian Daval}
\date{}
\begin{document}

\maketitle

\footnote{\keywordsubject{11A25, 11L07}}

\footnote{\keywords{Dirichlet divisor problem, exponential sums}}

\begin{abstract} Hardy showed that $\sum_{n \ioe x}\tau(n)-x(\log x +2\gamma -1)$ is not $o(x^{1/4})$. In this article, we prove that $\sum_{n \ioe x}\tau(n)(1-\frac{x}{n})-xP(\log x)=\frac{1}{4}+O \left( \frac{\log x}{x^{1/4}} \right)$, where $P$ is a  polynomial of degree 2. As a corollary, this estimate enables us to settle a conjecture surmised by Berkane, Bordell\`{e}s, and Ramar\'{e} dealing with the positivity of an integral of the error term in the Dirichlet divisor problem. All results are entirely explicit and allow us to study the proximity between the remainder of the Dirichlet divisor problem and its logarithmic version.
\end{abstract}

\section{Introduction}

\noindent
The multiplicative function which counts the number of divisors of $n$, denoted by $\tau(n)$ or sometimes $d(n)$,  is a classical subject of study in number theory.  By the convolution identity $\tau=\mathbf{1} \star \mathbf{1}$, we have $\sum_{n \ioe x}\tau(n)=\sum_{n \ioe x} \lfloor x/n \rfloor$, and therefore sums of fractional parts play a significant role in determining the order of magnitude of this sum. The ongoing conjecture is $\sum_{n \ioe x}\tau(n)-x(\log x +2\gamma -1)=O(x^{1/4+\varepsilon})$ for all $\varepsilon>0$, but seems to be currently out of reach. After applying the hyperbola principle, historically introduced by Dirichlet in this context, we obtain a shorter sum of fractional parts
\begin{align*}
   \sum_{n \ioe x}\tau(n)&=2\sum_{n \ioe \sqrt{x}} \left \lfloor \mfrac{x}{n} \right \rfloor-(\lfloor \sqrt{x} \rfloor)^2\\
   &=2xH(\sqrt{x})-2\sum_{n \ioe \sqrt{x}} \left( \left\lbrace  \mfrac{x}{n} \right\rbrace  - \mfrac{1}{2} \right) +\lfloor \sqrt{x} \rfloor-(\lfloor \sqrt{x} \rfloor)^2
\end{align*}
where $H$ is the harmonic sum. The heart of the divisor problem lies in estimating the sum 
\begin{equation}
\sum_{n \ioe \sqrt{x}} \psi \left( \mfrac{x}{n} \right) \label{somme courte} 
\end{equation}
where $\psi(x):=\{x\}-\frac{1}{2}$, and Voronoï showed that \eqref{somme courte} does not exceed $ \leqslant \np{1.5} \, x^{1/3} \log x + \np{5.5}$ for all $x \geqslant 1$. The best result to date is due to Li and Yang \cite{liyang23} who proved that, for $x$ sufficiently large and all $\varepsilon > 0$, we have
$$\sum_{n \ioe \sqrt{x}} \psi \left( \mfrac{x}{n} \right) \ll_{\varepsilon} x^{\theta + \varepsilon} $$
with $\theta = \frac{\np{1646}}{\np{6881}} + \frac{25\sqrt{\np{1717}}}{\np{13762}}\doteq \np{0.31448} \dotsc$, slightly improving on the previous record held by Huxley in \cite{hux03}. In \cite{ctrud21}, the following logarithmic version is investigated to obtain other explicit results with an application for class numbers of quartic number fields
\begin{equation} \label{CullyTrugdian}
    \bigg| \sum_{n \ioe x} \sdfrac{\tau(n)}{n} - \left( \tfrac{1}{2}\log^2 x +2\gamma \log x+\gamma^2-2\gamma_1 \right) \bigg|\ioe \mfrac{\np{1.001}}{x^{1/2}} \quad ( x \soe 2).
\end{equation}
Applying the hyperbola principle again, we derive
$$x \sum_{n \ioe x}\sdfrac{\tau(n)}{n} = 2x\sum_{n \ioe \sqrt{x}} \sdfrac{H(x/n)}{n}-x(H(\sqrt{x}))^2.$$
Under the conjecture $O(x^{1/4+\varepsilon})$ aforementioned, a summation by parts gives an error term equal to $O(x^{1/4+\varepsilon})$ for $\sum\tau(n)(1-x/n)$ and $\sum\tau(n)(x/n)$. But fractional parts also appear in the harmonic sum, since
\begin{align}
   H(x/n)(x/n)= (x/n)\log (x/n)+  \gamma x/n - \psi(x/n)+O\left( n/x \right).
\end{align}

\noindent
So the sum  $\sum_{n \ioe x}\tau(n)(1-x/n)$ contains no fractional part in the main term. The difficult part of the divisor problem is thus eliminated in this smooth version, and we shall prove below that the remainder is unconditionally $\frac{1}{4}+O( x^{-1/4} \log x)$. This error term is analysed in a first part using the Euler-Maclaurin formula, and in a second part with exponential sums. The number $\frac{1}{4}$ can be seen as a residue since the Dirichlet series of $\tau(n)$ is $\zeta(s)^2$ and $\zeta(0)^2=\frac{1}{4}$. 

\begin{thm}
\label{theoremdiviseur} 
For all real numbers $x \geqslant 1$, define
\begin{align*}
\Delta(x)&:=\sum_{n \ioe x}\tau(n)-x(\log x +2\gamma -1) \\
\text{and }\; \delta(x)&:=\sum_{n \ioe x} \sdfrac{\tau(n)}{n}-\left( \tfrac{1}{2}\log^2 x +2\gamma \log x+\gamma^2-2\gamma_1\right)\;.
\end{align*}
Then $\Delta(x)-x\delta(x)=  \mfrac{1}{4} + r(x)$ where
$$|r(x)| \ioe \begin{cases} \mfrac{1}{8}+\mfrac{\np{0.316}}{\sqrt{x}}+\mfrac{1}{64x},  & \text{ if } x \soe 1 \, ; \\ & \\ \mfrac{9\log x}{ x^{1/4}} +\mfrac{\np{0.239}}{\sqrt{x}}+\mfrac{1}{64x}, & \text{ if }  x \soe 170.
\end{cases}$$
\end{thm}

\noindent
Incidentally, this result solves unconditionally \cite[Conjecture~1.4]{ber12} with the exponent $\frac{7}{4}$ replaced by $2$ as proposed in \cite[Question~8.7]{ber12}. See  \cite[Section~8]{ber12} for more background on this conjecture. It should be pointed out \cite{ram26} that Mahatab and Mukhopadhyay have settled asymptotically this conjecture. See also \cite[Theorem~1.1]{math20} for a related topic. 

\begin{cor}
\label{corConjecture}
For all $x \geqslant 1$,
$$\int_x^{\infty} \sdfrac{\Delta(u)}{u^{2}} \diff{u} \geqslant 0.$$
\end{cor}

\noindent
It should be stressed that removing the sum \eqref{somme courte} is not sufficient to get the above result. To obtain the term $\frac{1}{4}+o(1)$, it is necessary to have at our disposal some effective results of the Chowla-Walum conjecture.          
Since the orders of magnitude in the Dirichlet divisors problem are between $x^{1/4}$ and $x^{1/2}$, we can convert explicit results between $\sum_{n \ioe x}\tau(n)$ and $\sum_{n \ioe x} \tau(n)/n$ with almost no loss. For example, Hardy's result $\Delta(x) \neq o(x^{1/4})$ transfers directly to the logarithmic version, whereas integration by parts would not have succeeded. Let us give another example. The left-hand side of \eqref{BBR} below comes from \cite[Lemma~7.1]{ber12}, so that Theorem~\ref{theoremdiviseur} yields at once the conversion

\begin{equation} \label{BBR}
    |\Delta(x)| \ioe  \np{0.397} \, x^{1/2} 
\Longrightarrow  \\
| \delta(x)|\ioe \mfrac{\np{0.397}}{x^{1/2}}+\mfrac{\np{0.38}}{x}  
\end{equation}
valid for all $ x \soe \np{5560} $,
without having to do any computer calculations. Furthermore, this improves the result \eqref{CullyTrugdian} by a factor $2$, while a summation by parts of \eqref{BBR} would only yield $3 \times \np{0.397}/x^{1/2}$.

\medskip

\noindent
To conclude this introduction, let us have a look if we replace the Dirichlet divisor function $\tau$ by the M\"obius function $\mu$. Let $M(t)=\sum_{n \ioe t} \mu(n)$ be the Mertens function, and $m(t)=\sum_{n \ioe t} \mu(n)/n$ be its logarithmic version. The Dirichlet series attached to the function $\mu$ is $1/\zeta(s)$, and we have the residue $1/\zeta(0)=-2$. Nevertheless, we have in this case
\begin{equation*}
    M(x)-xm(x)=x\int_{x}^{\infty} \sdfrac{M(u)}{u^2}\diff{u}=-2+r(x)
\end{equation*}
with $r(x) \neq o(x^{1/2})$ and $M(x)-xm(x)<0$ for $x \ioe \np{18349}$, as the residue suggests, but there is a first sign change between $\np{18349}$ and $\np{18350}$. These oscillations are due to the existence of nontrivial zeros of the Riemann zeta function with real part $\frac{1}{2}$. Regarding the proximity of the two summation functions, we have the following result in \cite{dav25}

$$\sdfrac{2}{3} \ioe  \sdfrac{\sup_{t \ioe x}|m(t)|t}{\sup_{t \ioe x }|M(t)|} \ioe \sdfrac{3}{2} \quad ( x \soe 94).$$

\section{Proofs}

\subsection{Notations}

\noindent
In what follows, we will use the first three Bernoulli polynomials $B_j(x)$:
\begin{small}
\begin{center}
\begin{tabular}{cccc}
$j$ & $1$ & $2$ & $3$ \\
& & & \\
$B_j(x)$ & $x - \frac{1}{2}$ & $x^2-x+\frac{1}{6}$ & $ x^3-\frac{3x^2}{2}+\frac{x}{2} $
\end{tabular}
\end{center}
\end{small}
and their attached periodic Bernoulli functions $B_j(\{x\})$. Notice that it is customary to denote $B_1 (\{x\}) := \psi(x)=\{x\}-\frac{1}{2}$ the $1$st periodic Bernoulli function. We also will use the classical exponential notation $e(x) := e^{2i \pi x}$, and $H(x):=\sum_{n \ioe x} \frac{1}{n}$ will be the harmonic sum. For all $x \geqslant 1$, define $R_1(x)$ and $R_2(x)$ to be respectively the error terms in the following asymptotic formulas
\begin{equation}
   H(x)=\log x+\gamma-\mfrac{\psi(x)}{x}+R_1(x) \label{eq:R_1}
\end{equation}
and
\begin{equation}
   \sum_{n \ioe x} \mfrac{1}{n} \log \left( \mfrac{x}{n}\right) = \tfrac{1}{2} \log^2 x +\gamma \log x-\gamma_1 +  R_2(x) \label{eq:R_2}
\end{equation}
where $\gamma \doteq \np{0.5772} \dotsc$ is the Euler-Mascheroni constant and $\gamma_1 \doteq - \np{0.0728} \dotsc$ is the $1$st Riemann-Stieltjes constant.
    
\subsection{Basic explicit estimates}

\subsubsection{Estimates using the Euler-Maclaurin summation formula}

\begin{lem}
\label{le:harmonic}
For all $x \geqslant 1$, we have
$$\left| R_1(x) \right| \leqslant \sdfrac{1}{8x^2}.$$
\end{lem}

\begin{proof}
Applying the Euler-Maclaurin summation formula \cite[p.~16 ]{bal16} we have the integral representation
$$R_1(x)=\int_{x}^{ \infty} \sdfrac{\psi(t)}{t^2}   \diff{t}.$$
Define $\psi_2(x) := \medint\int_{1}^{x}\psi(t)\diff{t} = \frac{1}{2}(\psi(x))^2-\frac{1}{8}$. Integrating by parts, we derive
\begin{equation}
    x^2R_1(x) = - \psi_2(x)+2x^2\int_{x}^{ \infty} \sdfrac{\psi_2(t)}{t^3} \diff{t}.
\end{equation}
Since $-\frac{1}{8} \ioe \psi_2(x) \ioe 0$, we get
$$-\mfrac{1}{8} \ioe 2x^2\int_{x}^{ \infty} \sdfrac{\psi_2(t)}{t^3} \diff{t} \ioe 0$$
which proves $x^2 |R_1(x)| \ioe \frac{1}{8}$ as claimed. 
\end{proof}

\begin{lem}
\label{le:log_harmonic}
For all $x \geqslant 1$, we have 
$$|R_2(x)| \ioe \sdfrac{\np{0.132}}{x^2}$$
and
$$ \left| (R_2-R_1)(x) \right| \ioe \sdfrac{\np{0.033}}{x^3}.$$
\end{lem}

\begin{proof}
By \cite[equations 10 and 11 p.~38, ex. 34 p.~41 (correction p.~113)]{bal16} we have 
$$R_2(x) = \int_{x}^{ \infty} \sdfrac{\psi(t)}{t^2} \left( 1+ \log\left( \mfrac{x}{t}\right) \right) \diff{t}$$
so that $x^2|R_2(x)| \ioe \frac{1}{8}(1+\mathrm{e^{-3}}) < \np{0.132}$ for all $x\soe 1$. Now, an integration by parts gives
$$(R_2-R_1)(x) = \int_{x}^{\infty} \sdfrac{\psi(t)\log(x/t)}{t^2}   \diff{t} = - \mfrac{1}{2} \int_{x}^{\infty}B_2(\{t\})g(t) \diff{t}$$
where $g(t):= \bigl( 1+2\log \left( \frac{x}{t}\right) \bigr)/t^3$. Following the book \cite{bal16}, we divide the integral so that the function $g$ is monotonic with a constant sign to apply the $2$nd mean value theorem. Since $\medint\int_{1}^{z}B_2(\{t\}) \diff{t}= \frac{1}{3} B_3 (\{z\})$ and  $c:=\sup_{z \soe 0}|B_3(\{z \})|=\mfrac{\sqrt{3}}{36}$, we derive $\left| \medint\int_{z_1}^{z_2}B_2(\{t\}) \diff{t} \right| \ioe \mfrac{2c}{3}$, and hence the $2$nd mean value theorem yields
 \begin{align*}
    & \Bigl|\int_{x}^{\mathrm{e}^{1/2}x} B_2(\{t\})g(t) \diff{t} \Bigr| \ioe \mfrac{2c}{3} \, g(x)=\mfrac{2c}{3x^3}  \\
    & \Bigl|\int_{\mathrm{e}^{1/2}x}^{\mathrm{e}^{5/2}x} B_2(\{t\})g(t) \diff{t} \Bigr| \ioe -\mfrac{2c}{3} \, g(\mathrm{e}^{5/2}x)=\mfrac{c(1+5\mathrm{e^{-5}})}{3x^3}   \\
    & \Bigl|\int_{\mathrm{e}^{5/2}x}^{\infty}B_2(\{t\})g(t) \diff{t} \Bigr| \ioe -\mfrac{2c}{3} \, g(\mathrm{e}^{5/2}x)=\mfrac{c(1+5\mathrm{e^{-5}})}{3x^3}.
 \end{align*}
Therefore, we finally get
$$x^3|(R_2-R_1)(x)| \ioe \mfrac{x^3}{2} \Bigl| \int_{x}^{\infty}B_2(\{t\})g(t) \diff{t} \Bigr| \ioe \mfrac{c}{3}(2+5\mathrm{e^{-5}}) \ioe \np{0.033}$$
 completing the proof.
 \end{proof}

\subsubsection{Estimates for the divisor function}

\noindent
The following result is \cite[Lemma~3.1]{ber12}.

\begin{lem}
\label{resteTau} For all $x \geqslant 1$, we have
$$\sum_{n \ioe x}\tau(n)=x(\log x +2\gamma -1)- 2\sum_{n \ioe \sqrt{x}}  \psi\left(\mfrac{x}{n}\right) + \mfrac{1}{4} + 2xR_1(\sqrt{x})-(\psi(\sqrt{x}))^2.$$
\end{lem}

\noindent
We now derive a similar estimate for the logarithmic sum.

\begin{lem}
\label{resteTauLog}
For all $x \geqslant 1$, we have
    \begin{align*}
       x\sum_{n \ioe x} \mfrac{\tau(n)}{n} &= x \left( \mfrac{1}{2} (\log x)^2 + 2 \gamma \log x + \gamma^2 - 2 \gamma_1\right) -x\left(R_1(\sqrt{x})\right)^2 + 2 x R_2(\sqrt{x}) \\
        & \hspace*{0.5cm} +2\sqrt{x}R_1(\sqrt{x})\psi(\sqrt{x})-\left(\psi(\sqrt{x})\right)^2  - 2\sum_{k \ioe  \sqrt{x}}\psi\left(\mfrac{x}{k}\right) + 2x\sum_{k \ioe  \sqrt{x}}\mfrac{1}{k}R_1\left(\mfrac{x}{k}\right).
    \end{align*} 
\end{lem}

\begin{proof}
Let $1\ioe y \ioe x$ be a parameter at our disposal. Dirichlet's hyperbola principle yields
\begin{equation*}
   \sum_{n \ioe x} \sdfrac{\tau(n)}{n}= \sum_{k \ioe y}\mfrac{1}{k}H\left(\mfrac{x}{k}\right)+\sum_{k \ioe x/y}\mfrac{1}{k}H\left(\mfrac{x}{k}\right)-H(y)H\left(\mfrac{x}{y}\right).
\end{equation*}
Now
   \begin{align*}
  \sum_{k \ioe y}\mfrac{1}{k}H\left(\mfrac{x}{k}\right) & = \sum_{k \ioe y} \mfrac{1}{k}\log\left(\mfrac{x}{k}\right)+\gamma H(y)-\mfrac{1}{x}\sum_{k \ioe y}\psi\left(\mfrac{x}{k}\right)+\sum_{k \ioe y}\mfrac{1}{k}R_1\left(\mfrac{x}{k}\right) \\
  &  = \sum_{k \ioe y} \mfrac{1}{k}\log\left(\mfrac{y}{k}\right)+\left(\log\left(\mfrac{x}{y}\right)+\gamma\right)H(y)-\mfrac{1}{x}\sum_{k \ioe y}\psi\left(\mfrac{x}{k}\right)+\sum_{k \ioe y}\mfrac{1}{k}R_1\left(\mfrac{x}{k}\right), 
   \end{align*}
and similarly
    $$\sum_{k \ioe x/y}\mfrac{1}{k}H\left(\mfrac{x}{k}\right) = \sum_{k \ioe x/y} \mfrac{1}{k}\log\left(\mfrac{x/y}{k}\right)+\left(\log y+\gamma\right)H\left(\mfrac{x}{y}\right) -\mfrac{1}{x}\sum_{k \ioe x/y}\psi\left(\mfrac{x}{k}\right)+\sum_{k \ioe x/y}\mfrac{1}{k}R_1\left(\mfrac{x}{k}\right).$$
Note that
\begin{multline*}
   H(y) \left(\log\left(\mfrac{x}{y}\right)+\gamma\right) +H\left(\mfrac{x}{y}\right)\left(\log y+\gamma\right)  -H(y)H\left(\mfrac{x}{y}\right) \\
   =(\log y +\gamma)\left(\log\left(\mfrac{x}{y}\right)+\gamma\right)-\mathfrak{R}_1(y)\mathfrak{R}_1\left(\mfrac{x}{y}\right)
\end{multline*}
where $\mathfrak{R}_1(t):=R_1(t)-\frac{\psi(t)}{t}$.  Hence
    \begin{align*}
       \sum_{n \ioe x} \sdfrac{\tau(n)}{n} &=(\log y +\gamma)\left(\log\left(\mfrac{x}{y}\right)+\gamma\right)-\mathfrak{R}_1(y)\mathfrak{R}_1\left(\mfrac{x}{y}\right) \\
        & \hspace*{1cm} + \sum_{k \ioe y} \mfrac{1}{k}\log\left(\mfrac{y}{k}\right)-\mfrac{1}{x}\sum_{k \ioe y}\psi\left(\mfrac{x}{k}\right)+\sum_{k \ioe y}\mfrac{1}{k}R_1\left(\mfrac{x}{k}\right) \\
        & \hspace*{2cm} + \sum_{k \ioe x/y} \mfrac{1}{k}\log\left(\mfrac{x/y}{k}\right)-\mfrac{1}{x}\sum_{k \ioe x/y}\psi\left(\mfrac{x}{k}\right)+\sum_{k \ioe x/y}\mfrac{1}{k}R_1\left(\mfrac{x}{k}\right).
    \end{align*}
   Choosing $y=\sqrt{x}$ yields 
    \begin{align*}
       x\sum_{n \ioe x} \mfrac{\tau(n)}{n} &= x(\log(\sqrt{x})+\gamma)^2-x\left(R_1(\sqrt{x})\right)^2+2\sqrt{x}R_1(\sqrt{x})\psi(\sqrt{x})-\left(\psi(\sqrt{x})\right)^2 \\
        & \hspace*{0.5cm} + 2x \sum_{k \ioe \sqrt{x}} \mfrac{1}{k}\log\left(\mfrac{ \sqrt{x}}{k}\right) - 2\sum_{k \ioe  \sqrt{x}}\psi\left(\mfrac{x}{k}\right) + 2x\sum_{k \ioe  \sqrt{x}}\mfrac{1}{k}R_1\left(\mfrac{x}{k}\right)
    \end{align*} 
  and we complete the proof by inserting \eqref{eq:R_2} in the first of the last three sums. 
\end{proof}

\begin{lem}
\label{partieR} 
For all $x \geqslant 1$, we have
\begin{equation*}
    \left| 2x \sum_{k \ioe  \sqrt{x}}\mfrac{1}{k} R_1\left(\mfrac{x}{k}\right) \right| \ioe \np{0.125}\left(1+\mfrac{1}{\sqrt{x}} \right).
\end{equation*}
\end{lem}

\begin{proof}
Lemma~\ref{le:harmonic} implies $|R_1(x/k)|\ioe \np{0.125} \, k^2/x^2$, and $\sum_{n=1}^{N}2n=N^2+N$ yields the bound $2\sum_{k \ioe \sqrt{x}}k \ioe x+\sqrt{x}$.
\end{proof}
 
\noindent
The largest contribution in the $O(1)$-term in  Theorem~\ref{theoremdiviseur} comes from the $\np{0.125}$ above, and the main aim of the next section is to replace it by a $o(1)$. This will be achieved in Lemma~\ref{corChowla} below.

\subsection{Explicit estimates of generalized Chowla-Walum sums}

\noindent
In this section, we will use the following notation: for $x \geqslant 1$, $\alpha > 1$, $\beta \in \RR_{\geqslant 0}$ and $j \in \ZZ_{\geqslant 1}$, set
$$G_{\alpha,\beta,j} (x) := \sum_{n \leqslant x^{1/\alpha}} n^{\beta} B_j \left( \left\lbrace \mfrac{x}{n} \right\rbrace \right).$$

\subsubsection{Results}

\begin{prop}
\label{pro:main}
Let $1 < \alpha < 3$, $\beta \in \RR_{\geqslant 1/2}$ and $j \in \ZZ_{\geqslant 2}$. Set $\Gamma_j := \mfrac{2 \eta(j)j!}{(2\pi)^{j}}$ where $\eta(j) = \zeta(j)$ if $j$ is even, $\eta(j) = 1$ otherwise. Then, for all $x \geqslant 1$, we have
$$\left| G_{\alpha,\beta,j} (x) \right| \leqslant \mfrac{2^{\beta + 5}}{\sqrt{\pi}}\Gamma_j  \Biggl( \Bigl( \zeta \left( j - \tfrac{1}{2}  \right) + \tfrac{\sqrt{\pi}}{2^{\beta+5}} \Bigr) x^{\frac{\beta}{\alpha} - \frac{1}{2\alpha} + \frac{1}{2}} + \zeta \left( j + \tfrac{1}{2} \right) x^{\frac{\beta}{\alpha} + \frac{3}{2\alpha} - \frac{1}{2}} \Biggr) \mathcal{L}_{\alpha,\beta}(x)$$
where $\mathcal{L}_{\alpha,\beta}(x) := \mfrac{3-\alpha}{2\alpha(\beta+1)} \mfrac{\log x}{\log 2} + 1$.
\end{prop}

\begin{rem}
When $\alpha \geqslant 3$, the trivial bound $\left| G_{\alpha,\beta,j} (x) \right| \leqslant \Gamma_j x^{\frac{\beta+1}{\alpha}}$ is better than the result above.
\end{rem}

\noindent
Applying to $(\alpha,\beta,j) = (2,1,2)$, we derive the next estimate.

\begin{cor}
\label{cor:main}
For all $x \geqslant 170$, we have
$$\left| \sum_{n \leqslant x^{1/2}} n B_2 \left( \left\lbrace \mfrac{x}{n} \right\rbrace \right) \right| < 9x^{3/4} \log x.$$
\end{cor}

\subsubsection{Tools}

\begin{lem}[Kusmin-Landau]
\label{le:K-L}
Let $N < N_1 \leqslant 2N$ be positive integers and $f \in C^1 \left[ N,N_1 \right]$ such that $f^{\, \prime}$ is non-decreasing and that there exist $c_1 \geqslant 1$ and $0 < \lambda_1 < 1$ such that, for all $x \in \left[ N,N_1 \right]$, we have
\begin{equation}
   k + \lambda_1 \leqslant f^{\, \prime}(x) \leqslant k+1- \lambda_1 \quad \left( k \in \ZZ \right). \label{eq:K-L}
\end{equation}
Then
$$\left | \sum_{N< n \leqslant N_1} e \left( \pm f(n) \right) \right | \leqslant \sdfrac{2}{\pi \lambda_1}.$$
\end{lem}

\begin{proof}
See \cite[Corollary~6.2]{bor20}. 
\end{proof}

\begin{lem}
\label{le:intermédiaire}
Let $N < N_1 \leqslant 2N$ be positive integers and $f \in C^2 \left[ N,N_1 \right]$ such that there exists a real number $\lambda_2 \in \left( 0,\pi^{-1} \right)$ such that, for all $x \in \left[ N,N_1 \right]$, we have $f^{\, \prime \prime}(x) \geqslant \lambda_2$. Assume that $f^{\, \prime}(x) \not \in \ZZ$ for all $x \in \left( N,N_1 \right)$. Then
$$\left | \sum_{N< n \leqslant N_1} e \left( \pm f(n) \right) \right | \leqslant \sdfrac{4}{\sqrt{\pi \lambda_2}}.$$
\end{lem}

\begin{proof}
First note that, under the hypotheses of the lemma, the function $f^{\, \prime}$ is continuous and strictly increasing on $\left[ N,N_1 \right]$, and hence is one-to-one. Define $u := f^{\, \prime}(N)$ and $v:=f^{\, \prime}(N_1)$, so that $f^{\, \prime}\left( \left[ N,N_1 \right] \right) = [u,v]$. Since $f^{\, \prime}(x) \not \in \ZZ$ for all $x \in \left( N,N_1 \right)$, $f^{\, \prime}\Bigl( \left( N,N_1 \right) \Bigr) \subseteq \Bigl( \lfloor u \rfloor,\lfloor u \rfloor + 1 \Bigr) $. Let $M_1,M_2 \in \left( N,N_1 \right)$ be real numbers such that $f^{\, \prime} (M_1) = u + t$ and  $f^{\, \prime} (M_2) = v - t$, where $t \in (0,1)$ is a parameter at our disposal. Split the sum into $3$ subsums:
$$\sum_{N< n \leqslant N_1} e \left( f(n) \right) = \sum_{N< n \leqslant M_1} e \left( f(n) \right) + \sum_{M_1< n \leqslant M_2} e \left( f(n) \right) + \sum_{M_2< n \leqslant N_1} e \left( f(n) \right)$$
and we treat the $1$st and $3$rd sums trivialy. The mean value theorem implies that there exists a real number $c \in \left( N,2N \right)$ such that
\begin{align*}
     \left | \sum_{N< n \leqslant M_1} e \left( f(n) \right) \right | & \leqslant \max \left( M_1-N,1 \right) = \max \left( \sdfrac{f^{\, \prime}(M_1)-f^{\, \prime}(N)}{f^{\, \prime \prime}(c)},1 \right) \\
     & \leqslant \max \left( \sdfrac{u + t -u}{\lambda_2} ,1 \right) = \max \left( \sdfrac{t}{\lambda_2},1 \right).
\end{align*}
Similarly, for some $d \in (N,2N)$, we derive
$$\left | \sum_{M_2< n \leqslant N_1} e \left( f(n) \right) \right | \leqslant \max \left( \sdfrac{f^{\, \prime}(N_1)-f^{\, \prime}(M_2)}{f^{\, \prime \prime}(d)},1 \right) \leqslant \max \left( \sdfrac{t}{\lambda_2},1 \right).$$
The $2$nd sum is treated with Lemma~\ref{le:K-L}. Note that $f^{\, \prime}$ is increasing and, for all $x \in \left[ M_1,M_2 \right]$, we have
$$\lfloor u \rfloor +t \leqslant u+t \leqslant f^{\, \prime}(x) \leqslant v-t \leqslant \lfloor u \rfloor + 1 - t,$$
so that Lemma~\ref{le:K-L} applies with \eqref{eq:K-L} used with $\lambda_1 = t$ and $k = \lfloor u \rfloor$, and yields
$$\left | \sum_{M_1< n \leqslant M_2} e \left( f(n) \right) \right | \begin{cases} = 0 , & \mathrm{if} \ \left \lfloor M_2 \right \rfloor - M_1 < 1 \, ; \\ & \\ \leqslant \max \left( 1 , \frac{2}{\pi t}\right) , & \mathrm{if} \ \left \lfloor M_2 \right \rfloor - M_1 \geqslant 1. \end{cases}.$$
Therefore, we deduce that
$$\left | \sum_{N< n \leqslant N_1} e \left( \pm f(n) \right) \right | \leqslant 2 \max \left( \sdfrac{t}{\lambda_2},1 \right) + \max \left( 1 , \sdfrac{2}{\pi t}\right)$$
and the required bound follows by choosing $t = \sqrt{\lambda_2 \pi^{-1}}$, and noticing that, with this choice of $t$ and the condition $\lambda_2 < \pi^{-1}$, we have $t < \pi^{-1}$ and hence $\mfrac{2}{\pi t} > 2$. 
\end{proof}

\noindent
We now derive an explicit version of the Van der Corput inequality.

\begin{lem}[van der Corput]
\label{le:crit_2_VdC}
Let $N < N_1 \leqslant 2N$ be positive integers and $f \in C^2 \left[ N,N_1 \right]$ such that there exist two real numbers $c_2 \geqslant 1$ and $\lambda_2 > 0$ such that, for all $x \in \left[ N,N_1 \right]$, we have
$$\lambda_2 \leqslant f^{\, \prime \prime}(x) \leqslant c_2 \lambda_2.$$
Then
$$\left | \sum_{N< n \leqslant N_1} e \left( \pm f(n) \right) \right | \leqslant \sdfrac{4}{\sqrt{\pi}} \left( c_2 N \lambda_2^{1/2} + 2 \lambda_2^{-1/2} \right).$$
\end{lem}

\begin{proof}
If $\lambda_2 \geqslant \pi^{-1}$, then
$$\sdfrac{4 N \lambda_2^{1/2}}{\sqrt{\pi}} \geqslant \sdfrac{4N}{\pi} > N \geqslant N_1 - N \geqslant \left | \sum_{N< n \leqslant N_1} e \left( f(n) \right) \right |$$
so that we may suppose $\lambda_2 < \pi^{-1}$. We pick the numbers $u$ and $v$ of the proof of Lemma~\ref{le:intermédiaire} and set
$$\left[ u,v \right] \cap \ZZ := \left\lbrace m+1,\dotsc,m+K \right\rbrace$$
where $m \in \ZZ$ et $K \in \ZZ_{\geqslant 1}$. For all integers $k \in \left\lbrace 1,\dotsc,K+1\right\rbrace$, define
$$J_k := \left( f^{\, \prime} \right)^{-1} \Bigl( \left( m+k-1,  m+k \right] \cap [u,v] \Bigr) \cap \ZZ.$$
Note that $f^{\, \prime} (x) \not \in \ZZ$ whenever $x \in \overset{\circ}{J_k}$, so that Lemma~\ref{le:intermédiaire} implies that
$$\left | \sum_{N< n \leqslant N_1} e \left( f(n) \right) \right | \leqslant \sum_{k=1}^{K+1} \left | \sum_{n \in J_k} e \left( f(n) \right) \right | \leqslant \sdfrac{4(K+1)}{\sqrt{\pi \lambda_2}}$$
and, by the mean value theorem, we have
$$K - 1 \leqslant v-u = f^{\, \prime}(N_1) - f^{\, \prime}(N) \leqslant c_2 (N_1-N) \lambda_2 \leqslant c_2 N \lambda_2$$
completing the proof.
\end{proof}

\begin{rem}
A quite similar result has been proven in \cite[Lemma~2.10]{pat21}, but with a different method. The coefficient of the main term of our result is slightly weaker, but we have improved that of the secondary term.
\end{rem}

\subsubsection{Proof of Proposition~\ref{pro:main}}
\noindent
We start by recording a well-known estimate for the periodic Bernoulli functions. 

\begin{lem}
\label{le:Bernoulli_bound}
Let $j \in \ZZ_{\geqslant 2}$ and set $\Gamma_j := \mfrac{2 \eta(j)j!}{(2\pi)^{j}}$ where $\eta(j) = \zeta(j)$ if $j$ is even, $\eta(j) = 1$ otherwise. Then
$$\sup_{x \in \RR} \left| B_j (\{x\}) \right| \leqslant \Gamma_j.$$
\end{lem}

\begin{proof}
When $j = 2h$ is even, $ \left | B_{2h} \left( \left\lbrace x \right\rbrace \right) \right | \leqslant |B_{2h}|$ for all $x \in \RR$ and $h \in \ZZ_{\geqslant 1}$, and the desired bound follows from \cite[(9.6) p.~17]{rad73}. When $j$ is odd, a result due to Lehmer (see \cite[Exercise~5.(e) p.~504]{monv07} for instance) shows that, for all $x \in \RR$ and $h \in \ZZ_{\geqslant 1}$, we have
$$\left | B_{2h+1} \left( \left\lbrace x \right\rbrace \right) \right | \leqslant \sdfrac{(2h+1)!}{2^{2h} \pi^{2h+1}}$$ 
and the asserted estimate follows.
\end{proof}

\noindent
We are now in a position to prove Proposition~\ref{pro:main}. 

\begin{proof}[Proof of Proposition~\ref{pro:main}]
For $x \geqslant 1$, $N \in \ZZ_{\geqslant 1}$, $\beta \geqslant 0$ and $j \in \ZZ_{\geqslant 2}$, define the partial sums
$$G_{N,\beta,j}(x) := \sum_{N < n \leqslant 2N} n^\beta B_j \left( \left\lbrace \mfrac{x}{n} \right\rbrace \right).$$
Let $T \leqslant x^{1/\alpha}$ be a parameter to be chosen later. Write
$$G_{\alpha,\beta,j} (x) = \left( \sum_{n \leqslant T} + \sum_{T < n \leqslant x^{1/\alpha}}\right) n^{\beta} B_j \left( \left\lbrace \mfrac{x}{n} \right\rbrace \right)$$
so that, setting $ L(x,T) := \mfrac{\log(x^{1/\alpha}/T)}{\log 2} + 1 $, we get
\begin{align}
   \left| G_{\alpha,\beta,j} (x) \right| & \leqslant \sup_{t \in \RR} \left| B_j (\{t\}) \right|  T^{\beta+1} + \left| \sum_{T < n \leqslant x^{1/\alpha}} n^{\beta} B_j \left( \left\lbrace \mfrac{x}{n} \right\rbrace \right) \right| \notag \\
   & \leqslant \Gamma_j T^{\beta+1} +  \max_{T < N \leqslant x^{1/\alpha}} \bigl| G_{N,\beta,j}(x) \bigr| L(x,T) \label{eq:main_1}
\end{align}
where we used Lemma~\ref{le:Bernoulli_bound} and where we splitted the interval $\left( T, x^{1/\alpha} \right]$ into dyadic intervals $(N,2N]$, whose number does not exceed $\leqslant L(x,T)$. Therefore, the problem amounts to estimating the sum $G_{N,\beta,j}(x)$. \\
\tri \: Since $j \geqslant 2$, we may expand the Bernoulli functions in Fourier series, the convergence being uniform: for all $t \in \RR$
   $$B_j \left( \{ t \} \right) = - \sdfrac{j!}{(2 \pi i)^j} \sum_{m \neq 0} \sdfrac{e(mt)}{m^j}.$$
Hence   
  $$G_{N,\beta,j}(x) = \pm \, \frac{2j!}{(2 \pi)^j} \left\lbrace \begin{array}{c} \re \\ \im \end{array} \right\rbrace \left( \sum_{m \geqslant 1} m^{-j} \sum_{N < n \leqslant 2N} n^\beta e \left(  \mfrac{mx}{n} \right)\right) \begin{array}{c} \mathrm{if} \ 2 \mid j \\ \mathrm{if} \ 2 \nmid j \end{array}$$ 
  so that
  $$\left| G_{N,\beta,j}(x) \right| \leqslant \frac{2j!}{(2 \pi)^j} \sum_{m \geqslant 1} m^{-j} \left| \sum_{N < n \leqslant 2N} n^\beta e \left(  \mfrac{mx}{n} \right)\right|$$
  and using the fact that $\Gamma_j \geqslant \frac{2j!}{(2 \pi)^j}$, we derive
  $$\left| G_{N,\beta,j}(x) \right| \leqslant \Gamma_j \sum_{m \geqslant 1} m^{-j} \left| \sum_{N < n \leqslant 2N} n^\beta e \left(  \mfrac{mx}{n} \right)\right|.$$
\tri \: By partial summation, we get
   $$\left| G_{N,\beta,j}(x) \right| \leqslant 2^{\beta + 1} \Gamma_j N^{\beta} \sum_{m \geqslant 1} m^{-j} \underset{N \leqslant N_1 \leqslant 2N}{\max} \left| \sum_{N < n \leqslant N_1} e\left( \sdfrac{mx}{n} \right)  \right|.$$
\tri \: Applying Lemma~\ref{le:crit_2_VdC} to the exponential sum with $\lambda_2 = \frac{mx}{4N^3}$ et $c_2 = 8$, we derive
   \begin{align}
      \left| G_{N,\beta,j}(x) \right| & \leqslant \mfrac{2^{\beta + 3}}{\sqrt{\pi}} \Gamma_j N^{\beta} \sum_{m \geqslant 1} m^{-j} \underset{N \leqslant N_1 \leqslant 2N}{\max} \Biggl(8 N \left( \sdfrac{mx}{4N^3}\right)^{1/2} + 2 \left( \sdfrac{4N^3}{mx}\right)^{1/2} \Biggr) \notag \\
      & = \mfrac{2^{\beta + 5}}{\sqrt{\pi}} \Gamma_j N^{\beta} \sum_{m \geqslant 1} m^{-j} \left( \sqrt{\sdfrac{mx}{N}} +  \sqrt{\sdfrac{N^{3}}{mx}} \, \right) \notag \\
      & \leqslant \mfrac{2^{\beta + 5}}{\sqrt{\pi}} \Gamma_j \biggl( \zeta (j - \tfrac{1}{2} ) x^{1/2} N ^{\beta - 1 /2} + \zeta (j + \tfrac{1}{2} ) x^{-1/2} N ^{\beta + 3 /2} \biggr). \label{eq:main_2}
   \end{align}
   Now assume $\beta \geqslant \frac{1}{2}$. Reporting \eqref{eq:main_2} in \eqref{eq:main_1} then yields
   $$\left| G_{\alpha,\beta,j} (x) \right|  \leqslant \Gamma_j T^{\beta+1} +   \mfrac{2^{\beta + 5}}{\sqrt{\pi}}\Gamma_j \biggl( \zeta (j - \tfrac{1}{2} ) x^{\frac{\beta}{\alpha} - \frac{1}{2\alpha} + \frac{1}{2}} + \zeta (j + \tfrac{1}{2} ) x^{\frac{\beta}{\alpha} + \frac{3}{2\alpha} - \frac{1}{2}} \biggr) L(x,T).$$
   Now choose $T = x^{\frac{1}{\beta+1} \left( \frac{\beta}{\alpha} - \frac{1}{2\alpha} + \frac{1}{2}\right) }$. Note that the condition $1 < \alpha < 3$ ensures that $T \leqslant x^{1/\alpha}$. This choice of $T$ yields the announced result.
\end{proof}

\subsubsection{Proof of Corollary~\ref{cor:main}}

\noindent
When $(\alpha,\beta,j) = (2,1,2)$, we have $\np{15.88} < \mfrac{64\Gamma_2}{\sqrt{\pi}} \left( \zeta \left( \frac{3}{2} \right) + \frac{\sqrt{\pi}}{64} \right) < \np{15.89}$ and $\np{8.07} < \mfrac{64\Gamma_2}{\sqrt{\pi}}  \zeta \left( \frac{5}{2} \right) < \np{8.08}$, so that
$$\left| G_{2,1,2} (x) \right| < \np{23.97} x^{3/4}\left(1 + \mfrac{\log x}{8 \log 2}\right) < 9x^{3/4} \log x$$
as soon as $x \geqslant 170$. \qed

\subsection{Finalization of the proof of Theorem~\ref{theoremdiviseur}}

\begin{lem}
\label{corChowla}
For all $x \geqslant 170$, we have
\begin{equation*}
    2x \left| \sum_{k \ioe  \sqrt{x}} \mfrac{1}{k} R_1\left(\mfrac{x}{k}\right) \right| \ioe \sdfrac{9\log x}{x^{1/4}} + \sdfrac{0.048}{x^{1/2}}.
\end{equation*}
\end{lem}

\begin{proof}
We proceed as in Lemmas~\ref{le:harmonic} and~\ref{le:log_harmonic}. By integration by parts, we have
  $$R_1(t )= - \sdfrac{B_2(\{t\})}{2t^2} + 2\int_{t}^{ \infty} \sdfrac{B_2(\{u\})}{u^3}   \diff{u} := - \sdfrac{B_2(\{t\})}{2t^2} + Q(t)$$
  so that
  $$2x \left| \sum_{k \ioe  \sqrt{x}}\mfrac{1}{k} R_1\left(\mfrac{x}{k}\right) \right| \ioe  \sdfrac{1}{x} \left| \sum_{k \ioe  \sqrt{x}} k B_2 \left( \left\lbrace \mfrac{x}{k} \right\rbrace \right) \right| + 2x \left|\sum_{k \ioe  \sqrt{x}} \mfrac{1}{k}Q\left( \mfrac{x}{k}\right) \right|.$$
By Corollary~\ref{cor:main}, the $1$st sum is $\ioe \mfrac{9\log x}{x^{1/4}}$, and by the $2$nd mean value theorem, we derive $\left|  Q(t) \right| \ioe \mfrac{4c}{3t^3}$ where $c= \mfrac{\sqrt{3}}{36}$, so that the $2$nd sum does not exceed
$$\ioe  \sdfrac{8c}{3x^2} \sum_{k \ioe \sqrt{x}} k^2 = \sdfrac{8c}{3x^2} \sdfrac{B_3(\sqrt{x}+1)-B_3(0)}{3}=:f(x).$$
Since $x \geqslant 170$, we get $\sqrt{x} \, f(x) \ioe \sqrt{170} \, f(170) \ioe \np{0.048}$, completing the proof.
\end{proof}

\begin{proof}[Proof of Theorem~\ref{theoremdiviseur}] Set $P(X) := -\frac{1}{2}X^2 - (2\gamma-1) X + 2(\gamma + \gamma_1) - \gamma^2 - 1$. By Lemmas~\ref{resteTau} and \ref{resteTauLog}, we get
\begin{align*}
    r(x) & := \sum_{n \ioe x} \tau(n) - x \sum_{n \ioe x} \sdfrac{\tau(n)}{n} - x P(\log x) - \mfrac{1}{4} \\
    & = x\left(R_1(\sqrt{x})\right)^2-2\sqrt{x}R_1(\sqrt{x})\psi(\sqrt{x})-2x\left( R_2 - R_1 \right) (\sqrt{x})-2x\sum_{k \ioe  \sqrt{x}}\mfrac{1}{k}R_1\left(\mfrac{x}{k}\right)
\end{align*}
and Lemma~\ref{partieR} applied to the last sum yields
\begin{align*}
   \left| r(x) \right| \ioe \sdfrac{\alpha^{2}}{x}+\sdfrac{\alpha}{\sqrt{x}}+ \sdfrac{2\beta}{\sqrt{x}}+\alpha\left(1+\mfrac{1}{\sqrt{x}}\right)=\alpha +\sdfrac{2(\alpha+\beta)}{\sqrt{x}}+\sdfrac{\alpha^{2}}{x}
\end{align*}
where  $\alpha:=\sup_{z \soe 1} |R_1(z)|z^2$ and $\beta:=\sup_{z \soe 1}|R_2(z)-R_1(z)|z^3$. Using Lemmas~\ref{le:harmonic} and~\ref{le:log_harmonic}, we get $\alpha \ioe \np{0.125}$ and $\beta \ioe \np{0.033}$, proving the $1$st estimate in Theorem~\ref{theoremdiviseur}. To prove the $2$nd part, we use Lemma~\ref{corChowla} instead of Lemma~\ref{partieR}, yielding
\begin{align*}
    \left| r(x) \right| \ioe \sdfrac{\alpha^{2}}{x}+\sdfrac{\alpha}{\sqrt{x}}+ \sdfrac{2\beta}{\sqrt{x}}+\sdfrac{9 \log x}{x^{1/4}}+\sdfrac{0.048}{x^{1/2}}= \sdfrac{9 \log x}{x^{1/4}}+\sdfrac{\alpha+2\beta+0.048}{\sqrt{x}}+\sdfrac{\alpha^{2}}{x}
\end{align*}
which provides the announced estimates.
\end{proof}

\begin{proof}[Proof of Corollary \ref{corConjecture}]
By summation by parts, we derive
$$ \sdfrac{\Delta(x)}{x}-\delta(x)=\int_{x}^{\infty} \sdfrac{\Delta(u)}{u^2} \diff{u}.$$
When $x\soe 7$, the first estimate of Theorem~\ref{theoremdiviseur} is less than $\np{0.247}$ and so 
    \begin{equation*}
        \Delta(x)-x\delta(x) \soe \mfrac{1}{4}-r(x) \soe \np{0.003} \;.
    \end{equation*}
To complete the proof in the case $1 \ioe x < 7$, we split the interval $\left[ 1,7 \right)$ into subintervals of the shape $I_m:=[m,m+1)$, $m = 1 , \dotsc , 6$, so that, for all $x \in I_m$, $\Delta(x)-x\delta(x) = Q_m(x)-xP( \log x)$, where $Q_m(x) := \sum_{n \ioe x} \tau(n) - x \sum_{n \ioe x} \frac{\tau(n)}{n}$ is a polynomial of degree $1$, and $P$ is the polynomial of degree $2$ aforementioned. For instance, if $m=5$, $Q_5(x) = -\frac{229}{60}x +10$ and therefore, for all $x \in \left[ 5,6 \right)$,
$$Q_5(x)-xP( \log x) = \tfrac{1}{2} x (\log x)^2 + (2\gamma-1) x \log x + x \left( \gamma^2 - 2 \gamma - 2 \gamma_1 - \tfrac{169}{60} \right) + 10$$
which is $\geqslant 0$ whenever $5 \leqslant x < 6$. There is no problem gluing the intervals together for the overall function is continuous. We use 40 decimal places for $\gamma$ and $\gamma_1$ thanks to \cite[Lemme~4]{ett19} or \cite[Exercise~7.2]{ram22}, which is more than sufficient to get the desired result in this case.
\end{proof}

\subsection*{Acknowledgments}

\noindent
The authors would like to express their sincere gratitude to Olivier Ramaré for bringing the unpublished results of Mahatab and Mukhopadhyay to their attention and for providing them with the references \cite{ett19} and \cite{ram22}.

\end{document}